\documentclass[letterpaper,11pt]{amsart}

\usepackage{epsfig}
\usepackage{amssymb}
\usepackage{amsfonts}
\usepackage{amsmath}
\usepackage{graphicx}
\usepackage{mathrsfs}
\usepackage{stmaryrd}
\usepackage{amsthm}
\usepackage[all]{xy}
\usepackage[top=1in, bottom=1in, left=1in, right=1in]{geometry}
\usepackage{color}
\usepackage[colorlinks=true]{hyperref}
\hypersetup{urlcolor=blue,citecolor=red,linkcolor=blue}

\newtheorem{theorem}{Theorem}[section]
\newtheorem{lemma}[theorem]{Lemma}
\newtheorem{proposition}{Proposition}[section]
\newtheorem{corollary}{Corollary}[section]

\newtheorem{remark}{Remark}[section]

\numberwithin{equation}{section}

\begin{document}
\date{\today}
\title
{Heintze-Karcher Type Inequalities for Fractional GJMS Operators}

\author{Huihuang Zhou}
\address{School of Mathematical Sciences\\
	University of Science and Technology of China\\
	Hefei, 230026\\ P.R. China\\}
\email{zhouhuihuang@ustc.edu.cn}


\begin{abstract}
In this paper, we genelize the Heintze-Karcher type inequalities for fractional Q-curvature $Q_{2\gamma}$ on conformally compact Einstein manifolds. Such inequality holds for all $\gamma\in (0,1]$. In particular, for $\gamma=\frac{1}{2}$ and $\gamma=1$, we obtain some rigidity theorems by characterising the equalities. 
\end{abstract}

\maketitle

\section{Introduction}
 
In 1978, Heintze-Karcher\cite{HK} proved some type of inequality for compact domain with smooth boundary in $\mathbb{R}^{n+1}$. The versions in $\mathbb{S}_{+}^{n+1}$ and $\mathbb{H}^{n+1}$ were obtained by Brendle\cite{Br} in 2014. Their results can be concluded as follows.
\begin{theorem}[Heintze-Karcher\cite{HK}, Brendle\cite{Br}]\label{thB}
Suppose $\Omega^{n+1} \subset \mathbb{R}^{n+1}$ (\textrm{respectively} $\mathbb{H}^{n+1}, \mathbb{S}^{n+1}_+$) is a compact domain with smooth boundary $M$. For $p\in \Omega$ fixed, let 
$$V(x)=1, \quad (\mathrm{respectively}\quad \cosh d_{\mathbb{H}}(x,p), \quad \cos d_{\mathbb{S}}(x,p)).$$ 
If the mean curvature $H$ of $M$ is positive, then
\begin{equation}\label{hk1}
\int_{M}\frac{V}{H}dS \geq \frac{n+1}{n}\int_{\Omega}V dV_{g}. 
\end{equation}
The equality in (\ref{hk1}) holds if and only if $\Omega$ is isometric to a geodesic ball.
\end{theorem}

In the case of more general manifolds, Heintze-Karcher type inequality has been studied by several authors. Montiel-Ros \cite{MR} and Ros\cite{Ro} showed that inequality (\ref{hk1}) remains true with $V=1$ for $(\Omega^{n+1} ,g)$ with non-negative Ricci curvature. Brendle\cite{Br} proved (\ref{hk1}) for more general warped product spaces. More results can be found in \cite{QX, LX, WX}. The Heintze-Karcher inequality has many interesting applications (see \cite{MR,Ro,PRS,BHW,LX,WX} for instance).


In \cite{QX}, the authors asked that whether Theorem \ref{thB} holds if $Ric\geq -ng$ due to the results of \cite{Ro} and \cite{QX}. 



\vspace{0.1in}
In this paper, we mainly want to generalize the Heintze-Karcher type inequalities for fractional Q-curvature on conformally compact Einstein manifolds. Suppose $X^{n+1}(n\geq 3)$ is a smooth compact manifold with boundary $M^n$. A function $\rho$ is a  defining function of $M$ in $X$ if
\[ 0\leq \rho\in C^{\infty}(\overline{X}),\quad \rho>0 \mathrm{~in~} X,\quad \rho=0 \mathrm{~on~} M,\quad d\rho|_{M}\neq 0.\]
We say $g_+$ is a conformally compact metric on $X$ with conformal infinity $(M,[\hat{g}])$ if there exists a defining function $\rho$ such that $(\overline{X},\bar{g})$ is compact for $\bar{g}=\rho^2 g_+$, and $\bar{g}|_{TM}\in [\hat{g}]$. Then if $g_+$ is Einstein, i.e. $Ric_{g_+}=-ng_+$, $(X,g_+)$ is called conformally compact Einstein manifold.
 
Given a conformally compact Einstein manifold $(X^{n+1},g_+)$ and a representative $\hat{g}\in [\hat{g}]$ on $M$, there is a uniquely geodesic normal defining function $r$ such that on $M\times [0,\epsilon)$ in $X$, $g_+$ has the normal form
\begin{equation}\label{geq}
 g_+ =r^{-2}(dr^2 + g_r),
\end{equation}
for $g_r$ a one-parameter family of metrics on $M$ with $g_0 =\hat{g}$ and having asymptotic expansion that is even in $r$ at least up to order $n$ according to \cite{CDLS}:
\begin{equation}\label{Geq}
g_r=
\begin{cases}
\hat{g} +r^2 g_2 +\cdots+r^{n-1}g_{n-1}+r^n g_n +O(r^{n+1}),& n \textrm{ is odd},\\
\hat{g} +r^2 g_2 +\cdots+r^{n-2}g_{n-2}+(r^n\log r)h+ r^n g_n +O(r^{n+1}\log r),& n \textrm{ is even}.
\end{cases}
\end{equation}
Here $g_i (1\leq i\leq n-1)$ and $h$ are symmetric 2-tensors determined by $\hat{g}$, and $g_n$ is the global term which can not be locally determined. 
%



Consider a fractional power $\gamma\in (0,\frac{n}{2})$ and denote $s=\frac{n}{2}+\gamma$, assume $\gamma\notin\mathbb{N}$ and $s(n-s)\notin \mathrm{Spec}(-\Delta_+)$, where $-\Delta_+$ is the Laplacian-Beltrami operator of $g_+$. It is well known (c.f.Mazzeo-Melrose\cite{MM}, Graham-Zworski\cite{GZ}) that, given any $f\in C^{\infty}(M)$, then there is a unique solution satisfying the following equation:
\begin{equation}\label{pos}
    -\Delta_+ u -s(n-s) u=0, \quad r^{s-n}u|_M =f,
\end{equation}
Moreover, $u$ takes the form
\[
u=r^{n-s}F+r^s G, \quad F|_M =f,\quad F,G\in C^{\infty}(\overline{X}).
\]
Then the scattering operator $S(s)$ is defined by
\[ S(s)f=G|_{M}.\]
And the fractional GJMS operator of order $2\gamma$ is defined by the renormalised scattering operator:
\[
P^{\hat{g}}_{2\gamma}:=d_\gamma S\left(\frac{n}{2}+\gamma\right),\quad \textrm{for } d_\gamma=2^{2\gamma}\frac{\Gamma(\gamma)}{\Gamma(-\gamma)}.
\]
Then the \emph{fractional Q-curvature} is 
\[Q^{\hat{g}}_{2\gamma}:=\frac{2}{n-2\gamma}P^{\hat{g}}_{2\gamma}(1).\]

This family of operators and fractional Q-curvature are intensively studied in \cite{CC,GZ,JS,MM}. Note that while $\gamma=k\leq \frac{n}{2}$ is a positive integer, it coincides with the classical GJMS operators of order $2k$. For example, if $\gamma=1$, then $P^{\hat{g}}_{2}$ is the conformal Laplacian, and $Q^{\hat{g}}_{2}$ is the scalar curvature up to a constant. We are aiming to understand the family properties of fractional Q-curvature, in particular, the Heintze-Karcher type inequalities. 

Next, we like to introduce the adapted compactification of $g_+$. Suppose $\mathrm{Spec}(-\Delta_+)>s(n-s)$, let $u_s$ satisfy the following equation: 
\begin{equation}\label{poss}
-\Delta_+ u_s -s(n-s)u_s =0, \quad r^{s-n}u_s |_M=1.
\end{equation}
Then according to \cite{CC}, the \emph{adapted compactification} of $g_+$ is defined by
    \[ \bar{g}_s=\rho_{s}^2 g_+, \quad \mathrm{where}\quad \rho_s=u_{s}^{\frac{1}{n-s}}.\]
In particular when $s=n+1$, the adapted compactification is also called \emph{Lee compactification}.

\vspace{0.1in}
Now we are ready to state the Heintze-Karcher type inequality of fractional Q-curvature.
\begin{theorem}\label{mth3}
Suppose $(X^{n+1},g_+)$ is a conformally compact Einstein manifold, the conformal infinity $(M,[\hat{g}])$ is of positive Yamabe type. Take $\hat{g}\in [\hat{g}]$ such that $R_{\hat{g}}>0$. Let $s=\frac{n}{2}+\gamma$ and $\bar{g}_s=\rho_{s}^2 g_+$ be the adapted compactificated metric, then for $\gamma\in (0,1)$, 
\begin{equation}\label{mhk3}
\int_{M}\frac{1}{(Q_{2\gamma})^{(1-\gamma)/\gamma}}dS_{\hat{g}}\geq C(n,\gamma)\int_{X} \left(\frac{1-|\bar{\nabla} \rho_s|^2}{\rho_{s}^{\gamma}}\right)^{(2\gamma-1)/\gamma}dV_{\bar{g}_s},
\end{equation}
where
\[
C(n,\gamma)=\frac{(n+2\gamma)^2}{4\gamma(n+1)}\left(-\frac{4\gamma}{d_{\gamma}}\right)^{(1-\gamma)/\gamma}.
\]
\end{theorem}

Note that the nonnegativity assumption on the Yamabe constant of the conformal infinity is used to satisfy the spectral condition required in (\ref{poss}).  In fact, by Lee\cite{Le}, if the boundary Yamabe constant $\mathcal{Y}(M,[\hat{g}])\geq 0$, then $\mathrm{Spec}(-\Delta_+)\geq \frac{n^2}{4}$.  Besides, we point out that the equality in (\ref{mhk3}) cannot hold except when $\gamma=\frac{1}{2}$ according to our proof. 

\begin{corollary}\label{cor}
Let $\gamma=\frac{1}{2}$ i.e. $ s=\frac{n+1}{2}$ in Theorem \ref{mth3}, then the mean curvature $\bar{H}$ on $M$ in terms of $\bar{g}_s$ satisfies $\bar{H}=nQ_1$ and we have the classical Heintze-Karcher type inequality
\begin{equation}\label{cla}
\int_{M}\frac{1}{\bar{H}}dS_{\hat{g}}\geq \frac{n+1}{n}\mathrm{Vol}(X,\bar{g}_s).
\end{equation}
The equality in (\ref{cla}) holds if and only if $(X,g_+)$ is isometric to the hyperbolic space.
\end{corollary}

For Lee compactification, we also derive a Heintze-Karcher type inequality of scalar curvature.
\begin{theorem}\label{mth4}
Suppose $(X^{n+1},g_+)$ is a conformally compact Einstein manifold, the conformal infinity $(M,[\hat{g}])$ is of positive Yamabe type. Take $\hat{g}\in [\hat{g}]$ such that $\hat{J}=\frac{1}{2(n-1)}R_{\hat{g}}>0$. Let $\bar{g}_L=\rho_{L}^2 g_+$ be the Lee compactificated metric, then
\begin{equation}\label{mhk4}
    \int_{M}\frac{1}{\hat{J}}dS_{\hat{g}}\geq \frac{2(n+1)}{n}\int_{X}\rho_L dV_{\bar{g}_L}.
\end{equation}
The equality in (\ref{mhk4}) holds if and only if $(X,g_+)$ is isometric to the hyperbolic space.
\end{theorem}

\vspace{0.1in}

The paper is outlined as follows: In Section \ref{ccem}, we introduce the adapted compactification and obtain an asymptotic Heintze-Karcher inequality in Proposition \ref{mth2}. In Section \ref{frac}, we collect some properties of fractional Q-curvature and rigidity result to prove Theorem \ref{mth3}, Corollary \ref{cor}. In a similar way, Theorem \ref{mth4} is proved in section \ref{scal}. 

\vspace{0.1in}
\textbf{Achknowlegement:} The author wants to thank Professor Fang Wang and Zuoqin Wang for valuable discussions and suggestions.

\vspace{0.2in}
\section{Conformally Compact Einstein Manifolds}\label{ccem}
Suppose $(X^{n+1}, g_+)$ is a conformally compact Einstein manifold with conformal infinity $(M,[\hat{g}])$ and for some defining function $\rho$,
\[\bar{g}=\rho^2 g_+, \quad \hat{g}=\bar{g}|_{TM}.\]
Let $R_{\bar{g}}$ and $R_{\hat{g}}$ be the scalar curvature of $\bar{g}$ and $\hat{g}$; $E_{\bar{g}}$ be the trace free part of the Ricci curvature $Ric_{\bar{g}}$; $A_{\bar{g}}$ be the Schouten tensor of $\bar{g}$:
\[A_{\bar{g}}=\frac{1}{n-1}\left( Ric_{\bar{g}}-\bar{J}\bar{g} \right),\quad  \bar{J}=\frac{R_{\bar{g}}}{2n}.\]
The conformal transformation implies the following relation between the two operators:
\[
\Delta_+ =\rho^2 \bar{\Delta}-(n-1)\rho\langle \bar{\nabla}\rho, \bar{\nabla}\cdot \rangle_{\bar{g}}
\]
and it also gives the relations of scalar curvature and Ricci curvature in terms of $\rho$:
\begin{equation}\label{ric}
R_{g_+}+ng_{+} =Ric_{\bar{g}}+(n-1)\rho^{-1}\bar{\nabla}^2 \rho +\rho^{-2}\left(\rho\bar{\Delta}\rho+n(1-|\bar{\nabla}\rho|_{\bar{g}}^2 \right)\bar{g},
\end{equation}
\begin{equation}
R_{g_+} +n(n+1)=\rho^{2}\left(\bar{R}+2n\rho^{-1}\bar{\Delta}\rho +n(n+1)\rho^{-2}(1-|\bar{\nabla}\rho|_{\bar{g}}^2)\right).
\end{equation}

Recall $r$ is the geodesic normal defining function in terms of $\hat{g}\in [\hat{g}]$. Let 
$$\bar{g}:=r^2 g_+,$$
then in a collar neighborhood $M\times [0,\epsilon)$,  $|\bar{\nabla}r|_{\bar{g}}=1$ and 
\begin{equation}\label{met}
\bar{g}=dr^2 +g_r =dr^2 +\hat{g}+r^2 g_2 +r^4 g_4+O(r^5).
\end{equation}
Direct computation shows that
\[ [g_2]_{ij} =-\hat{A}_{ij} , \quad [g_4]_{ij}=\frac{1}{4(n-4)}\left(-\hat{B}_{ij}+(n-4)\hat{A}_{i}^{k}\hat{A}_{jk}\right) .\]
where
\[
\hat{A}_{ij}=\frac{1}{n-2}\left(\hat{R}_{ij}-\frac{\hat{R}}{2(n-1)}\hat{g}_{ij}\right),\quad \hat{B}_{ij}=\hat{C}_{ijk,}^{\quad k}-\hat{A}^{kl}\hat{W}_{kijl},\quad \hat{C}_{ijk}=\hat{A}_{ij,k}-\hat{A}_{ik,j}.
\]

\vspace{0.1in}
\subsection{The adapted compactification}\label{ada}
For $\gamma\in (0,1)$, let $s=\frac{n}{2}+\gamma$ and $r$ be the geodesic defining function. Recall that $u_s$ satisfies the following equation: 
\begin{equation}
-\Delta_+ u_s -s(n-s)u_s =0, \quad r^{s-n}u_s |_M=1.
\end{equation}
Then $u_s$ is positive according to \cite{CS} if $\mathrm{Spec}(-\Delta_+)>s(n-s)$ and
\[
u_s =r^{n-s}\left(1+u_2 r^2 +O(r^4)\right) + r^{s}\left(u_{2\gamma}+O(r^2)\right),
\]
where
\[
u_2 =\frac{n-2\gamma}{8(1-\gamma)}\hat{J} ,\quad u_{2\gamma}=S(s)1=\frac{n-2\gamma}{2d_\gamma}Q_{2\gamma}.
\]
The adapted compactification of $g_+$ is defined by
\[ \bar{g}_s=\rho_{s}^2 g_+, \] 
where
\[ 
\rho_s=u_{s}^{\frac{1}{n-s}}=r\left(1+\rho_{2\gamma} r^{2\gamma}+O(r^{\alpha})\right),\quad \rho_{2\gamma}=\frac{1}{d_\gamma}Q_{2\gamma},\quad \alpha=\mathrm{min}\{4\gamma, 2\}.
\]
\begin{remark}\label{hq}
For $\gamma=\frac{1}{2}$, recall that
\[\bar{g}_s =\rho_{s}^2 g_+ =(1+\rho_1 r +O(r^2))(dr^2 +\hat{g}+r^2 g_2+ \cdots).\]
This implies that the second fundamental form $\Pi_{\bar{g}_s}=-\rho_1 \hat{g}$ in terms of the outward unit normal on the boundary and hence 
\begin{equation}
 \bar{H}=-n\rho_1 =nQ_{1},
\end{equation}
where $\bar{H}$ is the mean curvature on the boundary in terms of $\bar{g}_s$.
\end{remark}

\vspace{0.1in}
For $\gamma=\frac{n}{2}+1$, i.e. $s=n+1$, let $V$ satisfy the following equation:
\begin{equation}\label{leeV}
-\Delta_+ V +(n+1)V=0, \quad rV|_{M}=1.
\end{equation}
Then $V$ is positive and near the boundary
\begin{equation}\label{Uasy}
V=\frac{1}{r}\left(1+\frac{\hat{J}}{2n}r^2 +o(r^2)\right).
\end{equation}
The Lee compactification of $g_+$ is defined by
\[ \bar{g}_L =\rho_{L}^2 g_+,\]
where 
\[
\rho_{L}=V^{-1}=r\left(1-\frac{\hat{J}}{2n}r^2 +o(r^2)\right),
\]

\vspace{0.1in}
\subsection{Asymptotic Heintze-Karcher inequality}
At the end of this section, we prove an asymptotic Heintze-Karcher inequality on $(X^{n+1},g_+)$ which gives us some evidence of Heintze-Karcher type inequalities of fractional Q-curvature.

Denote
\[ X_r :=\{ p\in X |r(p)>r\}, \quad \partial X_r :=\{p\in X |r(p)=r \}.\]
By direct calculations, we have
\begin{proposition}\label{mth2}
Suppose $(X^{n+1},g_+)(n\geq 5)$ is a conformally compact Einstein manifold with conformal infinity $(M,[\hat{g}])$. Take $\hat{g}\in [\hat{g}]$ such that $E_{\hat{g}}\neq 0$. Let $r$ be the geodesic normal defining function in terms of $\hat{g}$. 
Then there exists $r_0 $ small enough such that for all $0< r<r_0<\epsilon$, 
\begin{equation}\label{mhk2}
\int_{\partial X_r}\frac{V}{H_r}dS \geq \frac{n+1}{n}\int_{X_r}VdV_{g_+},
\end{equation}
where $V$ satisfies (\ref{leeV}) and $H_r$ is the mean curvature on $\partial X_r$ with respect to $g_+$. 
\end{proposition}

\begin{proof}
At first for $r>0$ sufficiently small, the level sets $\partial X_r$ are smooth compact embedded hypersurfaces. It is straightforward to see from (\ref{met}) that
\begin{equation}\label{mean}
H_r =\frac{1}{2}Tr_{g_r}(-r\partial_r (r^{-2}g_r))=n-\frac{r}{2}Tr_{g_r}(\partial_r g_r)=n+\hat{J}r^2+\frac{1}{2}|\hat{A}|^{2}_{\hat{g}}r^4+O(r^5).
\end{equation}
Now by a direct application of the Taylor formula we have from (\ref{met}) and (\ref{mean}) that
\begin{equation}\label{det}
\sqrt{\frac{\mathrm{det}g_r}{\mathrm{det}\hat{g}}}=1-\frac{\hat{J}}{2}r^2 +\frac{1}{8}(\hat{J}^2 -|\hat{A}|^{2}_{\hat{g}})r^4 +O(r^{5}).
\end{equation}
Recall $V$ satisfies (\ref{leeV}) and near the boundary 
\begin{equation}\label{Veq}
 V=\frac{1}{r}\left(1+\frac{\hat{J}}{2n}r^2 +v_4 r^4 +O(r^5)\right),\quad v_4 =\frac{1}{8n(n-2)}\left(\hat{\Delta}\hat{J}-\hat{J}^2 +n|\hat{A}|^{2}_{\hat{g}}\right).
\end{equation}
Then
\begin{equation}\label{HV}
\frac{V}{H_r}=\frac{r^{-1}}{n}\left(1-\frac{1}{2n}\hat{J}r^2 +\alpha r^4 +O(r^5)\right),
\end{equation}
where
\[\alpha=v_4-\frac{1}{2n}|\hat{A}|^{2}_{\hat{g}} +\frac{1}{2n^2}\hat{J}^2.\]
Thus the integral of $\frac{V}{H_r}$ on $\partial X_r$ with respect to $g_+$ is
\[
\int_{\partial X_r}\frac{V}{H_r} dS_{g_+}=r^{-n}\int_{M}\frac{V}{H_r}\sqrt{\frac{\mathrm{det}g_r}{\mathrm{det}\hat{g}}}dS_{\hat{g}}.
\]
(\ref{det}) and (\ref{HV}) imply that
\begin{equation}\label{uh}
\int_{\partial X_r}\frac{V}{H_r} dS_{g_+}=\frac{r^{-n-1}}{n}\left(\mathrm{Vol}(M,\hat{g})+\alpha_1 r^2 +\alpha_2 r^4 +O(r^{5})\right),
\end{equation}
where
\[\alpha_1 =-\frac{n+1}{2n}\int_{M}\hat{J}dS_{\hat{g}},\quad \alpha_2 =\int_{M} \left(\alpha +\frac{1}{8}\left(\hat{J}^2 -|\hat{A}|^{2}_{\hat{g}}\right)+\frac{1}{4n}\hat{J}^2\right) dS_{\hat{g}}.\]
For the integral of $V$ over $X_r$ with respect to $g_+$, there exists a constant $C$ independent of $r$ such that 
\[
\int_{X_r}VdV_{g_+}=C+\int_{M}\int_{r}^{r_0}V\sqrt{\frac{\mathrm{det}g_t}{\mathrm{det}\hat{g}}}dtdS_{\hat{g}}.
\]
where $r_0<\epsilon$ is small enough.
Therefore similarly we have
\begin{equation}\label{v}
\int_{X_r}VdV_{g_+}=\frac{r^{-n-1}}{n+1}\left(\mathrm{Vol}(M,\hat{g})+\beta_1 r^2 +\beta_2 r^4 +O(r^{5})\right),
\end{equation}
where
\[
\beta_1 =-\frac{n+1}{2n}\int_{M}\hat{J}dS_{\hat{g}},\quad \beta_2 =\frac{n+1}{n-3}\int_{M}\left(v_4 +\frac{1}{8}(\hat{J}^2 -|\hat{A}|^{2}_{\hat{g}})-\frac{\hat{J}^2}{4n}\right)dS_{\hat{g}}.
\]
According to (\ref{uh}) and (\ref{v}), we finally have
\begin{equation}
\int_{\partial X_r}\frac{V}{H_r} dS_{g_+}=\frac{n+1}{n}\int_{X_r}VdV_{g_+} \cdot\left(1+ \beta r^4+O(r^{5})\right),
\end{equation}
with
\[
\beta =\frac{1}{n(n-2)^3}\frac{1}{\mathrm{Vol}(M,\hat{g})}\int_{M}|\hat{E}|^{2}_{\hat{g}}dS_{\hat{g}}.
\]
Therefore we have the inequality if $r<r_0$ is small enough
\[
\int_{\partial X_r}\frac{V}{H_r} dS_{g_+}\geq \frac{n+1}{n}\int_{X_r}VdV_{g_+}.
\]
\end{proof}

\vspace{0.2in}
\section{Fractional curvature $Q_{2\gamma}$}\label{frac}
In this section, $(X^{n+1}, g_+)$ is a conformally compact Einsterin manifold with conformal infinity $(M,[\hat{g}])$. Let $\bar{g}_s$ be the adapted compactificated metric in terms of $\hat{g}\in [\hat{g}]$. We first collect some properties of $Q_{2\gamma}(0<\gamma<1)$ and give some rigidity results for $(X^{n+1}, g_+)$, then we can proof Theorem \ref{mth3} and Corollary \ref{cor}.

For $\gamma \in (0,1)$, $\bar{g}_s =\rho_{s}^2 g_+$. Then $\rho_s$ satisfies the following equation
\begin{equation}\label{rhos}
\bar{\Delta}_s \rho_s =-s\rho_{s}^{2\gamma-1}T_s, \quad \textrm{with }\quad  T_s =\frac{1-|\bar{\nabla}\rho_s|^{2}_{\bar{g}_s}}{\rho_{s}^{2\gamma}}.
\end{equation}
and the scalar curvature of $\bar{g}_s$ is
\[ \bar{J}_s =\frac{2s-n-1}{2}\frac{1-|\bar{\nabla}\rho_{s}|^{2}_{\bar{g}_s}}{\rho_{s}^2}.\]

\begin{lemma}[Guillarmou-Qing\cite{GQ}]\label{qlem}
Suppose $\mathcal{Y}(M,[\hat{g}])>0$, $\hat{g}$ is a representative such that $R_{\hat{g}}>0$, then for all $\gamma\in (0,1)$, $Q_{2\gamma}^{\hat{g}}>0$.
\end{lemma}

Moreover if $Q_{2\gamma}^{\hat{g}}>0$, we can show that $T_s$ is also positive with $s=\frac{n}{2}+\gamma$. The proof of Lemma \ref{Tlem} is similar to \cite{CC, WZ2} where the authors proved that $\bar{J}_s >0$ for $\gamma>1$ if $\bar{J}_s |_M >0$.
\begin{lemma}\label{Tlem}
For $\gamma\in (0,1)$, suppose $\mathcal{Y}(M,[\hat{g}])>0$, $\hat{g}$ is a representative such that $R_{\hat{g}}>0$, then $T_s$ satisfies the following equation
\begin{equation}\label{Teq}
\begin{cases}
\Delta_s T_s +(2\gamma-1)\rho_{s}^{-1}\langle \bar{\nabla}\rho_s, \bar{\nabla}T_s\rangle_{\bar{g}_s}=-2\rho_{s}^{-2\gamma}\left|\bar{\nabla}^2 \rho_s -\frac{\bar{\Delta}_s\rho_s}{n+1}\bar{g}_s\right|_{\bar{g}_s}^2 + c(\gamma,n)T_{s}^2 \rho_{s}^{2\gamma-2},& \mathrm{in~}X;\\
T_s = -\frac{4\gamma}{d_\gamma}Q_{2\gamma},& \mathrm{on~} M.
\end{cases}
\end{equation}
where
\[ c(\gamma,n)=\frac{n(n+2\gamma)(2\gamma-1)}{2(n+1)}.\]
Moreover $T_s$ is positive in $X$.
\end{lemma}

\begin{proof} The following computation is done with respect to $\bar{g}_s$, we drop the subscript to simplify the presentation.
First, denote $\Phi$ a symmetric two tensor and
\[ \Phi := Ric_{g_+}+ng_{+}=Ric_{\bar{g}}+(n-1)\rho^{-1}\bar{\nabla}^2 \rho+\rho^{-2}\left(\rho\bar{\Delta}\rho+n(1-|\bar{\nabla}\rho|_{\bar{g}}^2 \right)\bar{g}. \]
Recall
\[ T=\frac{1-|\bar{\nabla}\rho|^2}{\rho^{2\gamma}},\]
then
\[
\begin{aligned}
\bar{\Delta}T
&=\rho^{-2\gamma}\bar{\Delta}(1-|\bar{\nabla}\rho|^2)+2\langle\bar{\nabla}(\rho^{-2\gamma}),\bar{\nabla}(1-|\bar{\nabla}\rho|^2)\rangle+(1-|\bar{\nabla}\rho|^2)\bar{\Delta}(\rho^{-2\gamma})\\
&=-\rho^{-2\gamma}\bar{\Delta}(|\bar{\nabla}\rho|^2)-4\gamma\rho^{-2\gamma-1}\langle\bar{\nabla}\rho, \bar{\nabla}(T\rho^{2\gamma})\rangle+T\rho^{2\gamma}\bar{\Delta}(\rho^{-2\gamma})\\
&=-\rho^{-2\gamma}\bar{\Delta}(|\bar{\nabla}\rho|^2)-4\gamma\rho^{-1}\langle\bar{\nabla}\rho,\bar{\nabla}T\rangle-2\gamma(2\gamma-1)T\rho^{-2}|\bar{\nabla}\rho|^2 -2\gamma T\rho^{-1}\bar{\Delta}\rho.
\end{aligned}
\]
By the Bochner formula and equation (\ref{rhos}), we have
\[
\begin{aligned}
\bar{\Delta}(|\bar{\nabla}\rho|^2)
&=2|\bar{\nabla}^2 \rho|^2 +2\langle\bar{\nabla}\bar{\Delta}\rho,\bar{\nabla}\rho\rangle+2Ric(\bar{\nabla}\rho,\bar{\nabla}\rho)\\
&=2|\bar{\nabla}^2 \rho|^2 -(n+2\gamma)\langle\bar{\nabla} (T\rho^{2\gamma-1}),\bar{\nabla}\rho\rangle+2Ric(\bar{\nabla}\rho,\bar{\nabla}\rho)\\
\end{aligned}
\]
and
\[
\langle\bar{\nabla} T,\bar{\nabla}\rho\rangle=-2\rho^{-2\gamma}\bar{\nabla}^2 \rho(\bar{\nabla}\rho,\bar{\nabla}\rho)-2\gamma T\rho^{-1}|\bar{\nabla}\rho|^2.
\]
Thus
\[
\begin{aligned}
\bar{\Delta}T+(2\gamma -1)\rho^{-1}\langle \bar{\nabla}\rho, \bar{\nabla}T\rangle_{\bar{g}}=
&-2\rho^{-2\gamma}\left|\bar{\nabla}^2 \rho -\frac{\bar{\Delta}\rho}{n+1}\bar{g}\right|^2 +\frac{n(n+2\gamma)(2\gamma-1)}{2(n+1)}T^2 \rho^{2\gamma-2}\\
&-2\rho^{-2\gamma}\Phi(\bar{\nabla}\rho,\bar{\nabla}\rho).
\end{aligned}
\]
Since $g_+$ is Einstein, we have
\[ \Phi(\bar{\nabla}\rho,\bar{\nabla}\rho)=0.\]
Thus we prove equation (\ref{Teq}).

According to Lemma \ref{qlem}, $T_s$ is positive on the boundary for all $\gamma\in (0,1)$. If $\gamma \leq \frac{1}{2}$, $T_s >0$ all over $\overline{X}$ by maximun principle. Let $\mathcal{E}\subset (0,1)$ be the subset such that for each $\gamma\in \mathcal{E} $ such that $T_s >0$ in $X$. It is obvious that $\mathcal{E} $ is nonempty and open. If $\gamma_0 \in (\frac{1}{2},1)$ is a limit point of $\mathcal{E} $, then $T_{s_0} \geq 0$ with $s_0=\frac{n}{2}+\gamma_0$. Recall $T_{s_0}$ satisfies equation (\ref{Teq}), if $\gamma_{0} \notin \mathcal{E} $, then there is an interior point $p\in X$ such that $T_{s_0}(p)=0$, which contradicts with the strong maximum principle. Hence $\gamma_0 \in \mathcal{E} $ and $\mathcal{E} $ is closed. Therefore for all $\gamma\in (0,1)$, $T_s $ is positive in $X$.  
\end{proof}

Besides, we also recall the uniqueness theorem of conformally compact Einstein manifolds from \cite{CWZ, WZ1}, which is first proved in \cite{CLW}.
\begin{lemma}\label{rig2}
Suppose $(X,g_+)$ is conforamlly compact Einstein manifold, $\bar{g}_s$ is the adapted compatification metric at $\gamma=\frac{1}{2}$. If $Ric_{\bar{g}_s}=0$ and $\bar{H}_{\bar{g}_s}$ is a constant, then $(\overline{X},\bar{g}_s)$ is isometric to the flat ball $(\mathbb{B}^n ,g_{\mathbb{R}})$ and $(X,g_+)$ is isometric to the hyperbolic space $\mathbb{H}^n$.
\end{lemma}


Now we are ready to prove Theorem \ref{mth3} and Corollary \ref{cor}.
\begin{proof}[Proof of Theorem \ref{mth3}]
First of all, recall $\gamma\in (0,1)$, $s=\frac{n}{2}+\gamma$, $\bar{g}_s =\rho_s g_+$ and
\[T_s=\frac{1-|\bar{\nabla}\rho_s|_{\bar{g}_s}^2}{\rho_{s}^{2\gamma}}, \quad T_s|_M = -\frac{4\gamma}{d_\gamma}Q_{2\gamma}.\]
Since $R_{\hat{g}}>0$, we know $Q_{2\gamma} >0$ by Lemma \ref{qlem} or Guillarmou-Qing\cite{GQ}, then $T_s >0$ in $X$ by Lemma \ref{Tlem}. Set
 $$k:=\frac{1-\gamma}{\gamma}.$$
Integrating by part, we have
\begin{multline}\label{green}
\int_{X} \rho_s \left(\bar{\Delta}_s T_{s}^{-k}+(2\gamma-1)\rho_{s}^{-1}\langle\bar{\nabla}\rho_s,\bar{\nabla}T_{s}^{-k}\rangle \right) dV_{\bar{g}_s}- (2-2\gamma)\int_{X}T_{s}^{-k}\bar{\Delta}_s\rho_s dV_{\bar{g}_s}\\
=\int_{M} \rho_s \frac{\partial T_{s}^{-k}}{\partial \nu}-(2-2\gamma)T_{s}^{-k}\frac{\partial\rho_s}{\partial\nu}dS_{\hat{g}},
\end{multline}
where $\nu$ is unit outerward normal vector in terms of $\bar{g}_s$.
By equation (\ref{rhos}), we have
\begin{equation}\label{r-1}
T_{s}^{-k}\bar{\Delta}_s \rho_s =-(\frac{n}{2}+\gamma)\rho_{s}^{2\gamma-1}T_{s}^{1-k}.
\end{equation}
Equation (\ref{Teq}) implies that
\begin{multline}\label{t-1}
\bar{\Delta}_s T_{s}^{-k}+(2\gamma-1)\rho_{s}^{-1}\langle\bar{\nabla}\rho_s,\bar{\nabla}T_{s}^{-k}\rangle
=\\ 2k\rho_{s}^{-2\gamma}T_{s}^{-k-1}\left|\bar{\nabla}^2 \rho_s -\frac{\bar{\Delta}_s\rho_s}{n+1}\bar{g}_s\right|_{\bar{g}_s}^2 +k(k+1)T^{-k-2}_{s}|\bar{\nabla}T_s|^{2}_{\bar{g}_s}- c(\gamma,n)k \rho_{s}^{2\gamma-2}T_{s}^{1-k}.
\end{multline}
According to (\ref{r-1}) and (\ref{t-1}), the left side of equation (\ref{green}) becomes
\begin{equation}
\begin{aligned}
L.H.S=&\int_{X}2k\rho_{s}^{1-2\gamma}T_{s}^{-k-1}\left|\bar{\nabla}^2 \rho_s -\frac{\bar{\Delta}_s\rho_s}{n+1}\bar{g}_s\right|_{\bar{g}_s}^2 +k(k+1)\rho_s T^{-k-2}_{s}|\bar{\nabla}T_s|^{2}_{\bar{g}_s} dV_{\bar{g}_s}\\
&+ \left[(1-\gamma)(n+2\gamma)-c(\gamma,n)k\right]\int_{X} \rho_{s}^{2\gamma-1}T_{s}^{1-k} dV_{\bar{g}_s}
\end{aligned}
\end{equation}
Since $\rho_s |_{M}=0$, $\frac{\partial\rho_s}{\partial\nu}=-1$ and $T_{s}|_M =-\frac{4\gamma}{d_{\gamma}}Q_{2\gamma}$, the right side of equation (\ref{green}) becomes
\[
R.H.S=(2-2\gamma)\left(-\frac{4\gamma}{d_{\gamma}}\right)^{-k}\int_{M} Q_{2\gamma}^{-k}dS_{\hat{g}}.
\]
Therefore we obtain the inequality with $k=\frac{1-\gamma}{\gamma}$
\begin{equation}\label{trho}
\int_{M} Q_{2\gamma}^{-k}dS_{\hat{g}}\geq C(n,\gamma)\int_{X} \rho_{s}^{2\gamma-1}T_{s}^{1-k} dV_{\bar{g}_s},
\end{equation}
where 
$$C(n,\gamma)=\frac{(n+2\gamma)^2}{4\gamma(n+1)}\left(-\frac{4\gamma}{d_{\gamma}}\right)^{k},\quad d_{\gamma}=2^{2\gamma}\frac{\Gamma(\gamma)}{\Gamma(-\gamma)}<0.$$
\end{proof}

It is easy to see that the equality in (\ref{trho}) does not hold if $\gamma\neq \frac{1}{2}$.

\begin{proof}[Proof of Corollary \ref{cor}]
When $\gamma=\frac{1}{2}$, $s=\frac{n+1}{2}$, the ineuqality (\ref{trho}) becomes
\begin{equation}\label{class}
\int_{M}\frac{1}{\bar{H}}dS_{\hat{g}}\geq \frac{n+1}{n}\mathrm{Vol}(X,\bar{g}_s),
\end{equation}
which is the classical Heintze-Karcher inequality since $Q_1 =\frac{\bar{H}}{n}$ by Remark \ref{hq}.
If the equality in (\ref{class}) holds, we immediately have $T_s$ is constant and
\[ \bar{\nabla}^2 \rho_s -\frac{\bar{\Delta}_s\rho_s}{n+1}\bar{g}_s =0.\]
Thus the mean curvature on the boundary $\bar{H}$ is also constant. Moreover by conformal transformation in (\ref {ric}) and equation (\ref{rhos}) at $s=\frac{n+1}{2}$, we get
\[
Ric_{\bar{g}_s}=-(n-1)\rho_{s}^{-1}\left( \bar{\nabla}^2 \rho_s -\frac{\bar{\Delta}_s \rho_s}{n+1}\bar{g}_s \right)=0.
\]
Therefore according to Lemma \ref{rig2}, $(\overline{X},\bar{g})$ is isometric to an Euclidean ball and $(X,g_+)$ is isometric to Hyperbolic space.
\end{proof}

\vspace{0.2in}
\section{Scalar curvature $\hat{J}$}\label{scal}
Suppose $(X^{n+1}, g_+)$ is a conformally compact Einsterin manifold with conformal infinity $(M,[\hat{g}])$. Let $\bar{g}_L =\rho_{L}g_+$ be the adapted compactificated metric in terms of $\hat{g}\in [\hat{g}]$. Then in $\overline{X}$,
\begin{equation}\label{leq}
\bar{\Delta}_L \rho_L =-2\rho_{L}\bar{J}_L,\quad \textrm{with}\quad \bar{J}_L =\frac{n+1}{2}\frac{1-|\bar{\nabla}\rho_L|^{2}_{\bar{g}_L}}{\rho_{L}^2}.
\end{equation}
Moreover $\bar{J}_L >0$ if $\bar{J}_L |_{M}>0$, which is first proved in \cite{Le}.
\begin{lemma}\label{Jlem}
For $s=n+1$, suppose that $R_{\hat{g}}>0$ in terms of the representative $\hat{g}\in [\hat{g}]$, then $\bar{J}_L$ satisfies the following equation
\begin{equation}\label{jeq}
\begin{cases}
\Delta_L \bar{J}_L -(n-1)\rho_{L}^{-1}\langle \bar{\nabla}\rho_L,\bar{\nabla} \bar{J}_L\rangle_{\bar{g}_L}=-(n+1)\rho_{L}^{-2}\left|\bar{\nabla}^2 \rho_L -\frac{\bar{\Delta}_L \rho_L}{n+1}\bar{g}_L\right|_{\bar{g}_L}^2 ,& \mathrm{in~}X;\\
\bar{J}_L = \frac{n+1}{n}\hat{J},& \mathrm{on~} M.
\end{cases}
\end{equation}
Moreover $\bar{J}_L $ is positive all over $\overline{X}$. 
\end{lemma}
\begin{proof}
Similarly the following computation is done with respect to $\bar{g}_L$, we drop the subscript to simplify the presentation. Denote $\Psi$ a symmetric two tensor and 
\[
\Psi:=Ric_{g_+}+ng_{+}=Ric_{\bar{g}}+(n-1)\rho^{-1}\bar{\nabla}^2 \rho+\rho^{-2}\left(\rho\bar{\Delta}\rho+n(1-|\bar{\nabla}\rho|_{\bar{g}}^2) \right)\bar{g}.
\]
Recall
\[ \bar{J}=\frac{n+1}{2}\frac{1-|\bar{\nabla}\rho|^2}{\rho^{2}},\]
then
\[
\begin{aligned}
\bar{\Delta}\bar{J}
&=\frac{n+1}{2}\left(\rho^{-2}\bar{\Delta}(1-|\bar{\nabla}\rho|^2)+2\langle\bar{\nabla}(\rho^{-2}),\bar{\nabla}(1-|\bar{\nabla}\rho|^2)\rangle+(1-|\bar{\nabla}\rho|^2)\bar{\Delta}(\rho^{-2})\right)\\
&=-\frac{n+1}{2}\rho^{-2}\bar{\Delta}(|\bar{\nabla}\rho|^2)-4\rho^{-3}\langle\bar{\nabla}\rho, \bar{\nabla}(\bar{J}\rho^{2})\rangle+\bar{J}\rho^{2}\bar{\Delta}(\rho^{-2})\\
&=-\frac{n+1}{2}\rho^{-2}\bar{\Delta}(|\bar{\nabla}\rho|^2)-4\rho^{-1}\langle\bar{\nabla}\rho, \bar{\nabla}\bar{J})\rangle-2\bar{J}\rho^{-1}\bar{\Delta}\rho-2\bar{J}\rho^{-2}|\bar{\nabla}\rho|^2.
\end{aligned}
\]
By the Bochner formula and equation (\ref{leq}), we have
\[
\begin{aligned}
\bar{\Delta}(|\bar{\nabla}\rho|^2)
&=2|\bar{\nabla}^2 \rho|^2 +2\langle\bar{\nabla}\bar{\Delta}\rho,\bar{\nabla}\rho\rangle+2Ric(\bar{\nabla}\rho,\bar{\nabla}\rho)\\
&=2|\bar{\nabla}^2 \rho|^2 -4\langle\bar{\nabla} (J\rho),\bar{\nabla}\rho\rangle+2Ric(\bar{\nabla}\rho,\bar{\nabla}\rho),
\end{aligned}
\]
and
\[
\langle\bar{\nabla} \bar{J},\bar{\nabla}\rho\rangle=-(n+1)\rho^{-2}\bar{\nabla}^2 \rho (\bar{\nabla}\rho,\bar{\nabla}\rho)-2\bar{J}\rho^{-1}|\bar{\nabla}\rho|^2.
\]
Thus
\[
\bar{\Delta}\bar{J}-(n-1)\rho^{-1}\langle \bar{\nabla}\rho, \bar{\nabla}T\rangle_{\bar{g}}=
-(n+1)\rho^{-2}\left|\bar{\nabla}^2 \rho -\frac{\bar{\Delta}\rho}{n+1}\bar{g}\right|^2 
-(n+1)\rho^{-2}\Psi(\bar{\nabla}\rho,\bar{\nabla}\rho).
\]
Since $g_+$ is Einstein, we have
\[ \Psi(\bar{\nabla}\rho,\bar{\nabla}\rho)=0.\]
Thus equation (\ref{jeq}) holds. Since $\hat{J}>0$, we have $\bar{J}_L >0$ by maximum principle.  
\end{proof}

Now we are ready to prove Theorem \ref{mth4}.
\begin{proof}[Proof of Theorem \ref{mth4}]
At first, recall $\bar{g}_L =\rho_L g_+$ and
\[\bar{J}_L=\frac{n+1}{2}\frac{1-|\bar{\nabla}\rho_{L}|^2}{\rho_{L}^{2}}, \quad \bar{J}_L|_M = \frac{n+1}{n}\hat{J}.\]
According to Lemma (\ref{Jlem}), we know that $\bar{J}_L >0$ in $X$.
By direct computation, we have
\begin{multline}\label{eq41}
\int_{X}~ \rho_L \left(\bar{\Delta}_L \bar{J}_{L}^{-1}-(n-1)\rho_{L}^{-1}\langle \bar{\nabla}\rho_L,\bar{J}_{L}^{-1}\rangle_{\bar{g}_L}\right)-n\bar{J}_{L}^{-1}\bar{\Delta}_L \rho_L ~dV_{\bar{g}_L}
\\=\int_{M}\left(\rho_L \frac{\partial \bar{J}_{L}^{-1}}{\partial\nu}-n\bar{J}_{L}^{-1}\frac{\partial\rho_L}{\partial\nu}\right)dS_{\hat{g}}=\frac{n^2}{n+1}\int_{M}\frac{1}{\hat{J}}dS_{\hat{g}}.
\end{multline}
where $\nu$ is unit outerward normal vector in terms of $\bar{g}_L$.
From equation (\ref{leq}), we have
\begin{equation}\label{eq42}
\bar{J}_{L}^{-1}\bar{\Delta}_L \rho_L =-2\rho_L.
\end{equation}
Equation (\ref{jeq}) tells us that
\begin{equation}\label{eq43}
\bar{\Delta}_L \bar{J}_{L}^{-1}-(n-1)\rho_{L}^{-1}\langle \bar{\nabla}\rho_L,\bar{\nabla}\bar{J}_{L}^{-1}\rangle_{\bar{g}_L}=2\bar{J}_{L}^{-3}|\bar{\nabla}\bar{J}_{L}|_{\bar{g}_L}^2+ (n+1)\rho_{L}^{-2}\bar{J}_{L}^{-2}\left|\bar{\nabla}^2 \rho_L -\frac{\bar{\Delta}_L \rho_L}{n+1}\bar{g}_L\right|^{2}_{\bar{g}_L}.
\end{equation}
Therefore, together equation (\ref{eq41}) with (\ref{eq42}) and (\ref{eq43}), we have
\begin{multline}\label{eq44}
\frac{n^2}{n+1}\int_{M}\frac{1}{\hat{J}}dS_{\hat{g}}=2n\int_{X}\rho_L dV_{\bar{g}_L}
\\
+\int_{X}\left(2\rho_L \bar{J}_{L}^{-3}|\bar{\nabla}\bar{J}_{L}|_{\bar{g}_L}^2+(n+1)\rho^{-1}\bar{J}_{L}^{-2}\left|\bar{\nabla}^2 \rho_L -\frac{\bar{\Delta}_L \rho_L}{n+1}\bar{g}_L\right|^{2}_{\bar{g}_L}\right)dV_{\bar{g}_L}.
\end{multline}
Thus we obtain the inequality (\ref{mhk4}).

If the equality in (\ref{mhk4}) holds, from (\ref{eq44}) we immediately have $\bar{J}_L = \frac{n+1}{n}\hat{J}$ is constant and
\[ 
\bar{\nabla}^2 \rho_L -\frac{\bar{\Delta}_L \rho_L}{n+1}\bar{g}_L=0.
\]
WLOG, let $\bar{J}_L = \frac{n+1}{n}\hat{J}=\frac{n+1}{2}$, then
\[
\bar{\Delta}_L \rho_L =-2\bar{J}_L \rho_L=-(n+1)\rho_L,
\]
and
\[
\bar{\nabla}^2 \rho_L +\rho_L\bar{g}_L=0, \quad \rho_{L}|_{M}=0.
\]
By an Obata type result in \cite{Re}, $(\overline{X},M, \bar{g}_L)$ must be isometric to the half sphere $(\mathbb{S}_{+}^{n+1},\mathbb{S}^{n}, g_{S^{n+1}})$. Thus $(X,g_+)$ is isometric to Hyperbolic space due to \cite{Qi,ST, DJ, LQS, CLW}.
\end{proof}


\begin{thebibliography}{99}

\bibitem[Br]{Br}
Brendle, S. {Constant mean curvature surfaces in warped product manifolds}. 
Publications Math\'{e}matiques. Institut de Hautes \'{E}tudes Scientifiques, 2013,{117(1)}: 247–269.

\bibitem[BHW]{BHW}
Brendle S, Hung P, and Wang M. {A Minkowski-type inequality for hypersurfaces in the
anti-deSitter-Schwarzschild manifold}. Comm. Pure Appl. Math.,2016,{69(1)}: 124–144.

\bibitem[CC]{CC}
Case J, Chang S-Y A. 
{On the fractional GJMS operators}.
Comm. Pure Appl. Math., 2016, {69(6)}: 1017-1061.

\bibitem[CDLS]{CDLS}
P. T. Chru\'{s}ciel, E. Delay, J. M. Lee, and D. N. Skinner. 
{Boundary regularity of conformally compact Einstein metrics}. 
J. Differential Geom., 2005, {69(1)}:111–136.

\bibitem[CLW]{CLW}
Chen X, Lai M, Wang F. {Escobar-Yamabe compactifications for Poincare-Einstein manifolds and rigidity theorems}. Adv Math, 2019, {343}: 16-35.

\bibitem[CS]{CS}
D. Fischer-Colbrie, R. Schoen. {The structure of complete stable minimal surfaces in 3-manifolds of nonnegative scalar curvature}. Comm. Pure Appl. Math.,1980,{33(2)}:199–211.

\bibitem[CWZ]{CWZ}
Chen D, Wang F, Zhang X. Eigenvalue Estimate of the Dirac operator and Rigidity of Poincare-Einstein Metrics.
Math Z, 2019, 293: 485-502.

\bibitem[DJ]{DJ}
Dutta S, Javaheri M. {Rigidity of conformally compact manifolds with the round sphere as conformal infinity}.
Adv Math, 2010, {224}: 525-538.

\bibitem[FG]{FG}
C. Fefferman and C. R. Graham. 
{The ambient metric, volume 178 of Annals of Mathematics Studies}. 
Princeton University Press, Princeton, NJ, 2012.

\bibitem[GQ]{GQ}
Guillarmou C, Qing J. 
{Spectral Characterization of Poincar\'{e}-Einstein manifolds with infinity of positive Yamabe
type}. Int Math Res Notices, 2010, {9}: 1720-1740.

\bibitem[GZ]{GZ}
Graham C R, Zoworski M. 
{Scattering matrix in conformal geometry}. 
Invent Math, 2003, {152}: 89-118.

\bibitem[HK]{HK}
Heintze, E. and H. Karcher. 
{A general comparison theorem with applications to volume estimates for submanifolds}. 
Annales Scientifiques de l'\'{E}cole Normale Sup\'{e}rieure, 1978,{11(4)}: 451–70.

\bibitem[JS]{JS}
Joshi M, S\'{a} Barreto A. {Inverse scattering on asymptotically hyperbolic manifolds}. Acta Math, 2000, {184}: 41-86.

\bibitem[Le]{Le}
 Lee J. {The spectrum of an asymptotically hyperbolic Einstein manifold}. Comm Anal Geom, 1995, {3(1-2)}:
253-271.

\bibitem[LQS]{LQS}
Li G, Qing J, Shi Y. {Gap phenomena and curvature estimates for conformally compact Einstein manifolds}.
Trans Amer Math Soc, 2017, {369(6)}: 4385-4413.

\bibitem[LX]{LX}
Li J, Xia C.{An integral formula and its applications on sub-static manifolds}. J. Differential Geometry, 2019,{113}:493-518.

\bibitem[MM]{MM}
Mazzeo R, Melrose R.
{Meromorphic extension of the resolvent on complete spaces with asymptotically constant
negative curvature}. 
J Funct Anal, 1987, {75}: 260-310.

\bibitem[MR]{MR}
Montiel, S. and A. Ros. {Compact Hypersurfaces: The Alexandrov Theorem for Higher-Order
Mean Curvatures}. 
Pitman Monographs and Surveys in Pure and Applied Mathematics,in honor of M.P. do Carmo; edited by B. Lawson and K. Tenenblat,1991,{52}:279-296.

\bibitem[PRS]{PRS}
Pigola, S., M. Rigoli and A. G. Setti. {Some applications of integral formulas in Riemannian
geometry and PDE's.} Milan Journal of Mathematics 71, 2003,{1}: 219–281.

\bibitem[Qi]{Qi}
Qing J. {On the rigidity for conformally compact Einstein manifolds}. Int. Math. Res. Notices, 2003, {21}: 1141-1153.

\bibitem[QX]{QX}
G. Qiu and C. Xia 
{A Generalization of Reilly’s Formula and its Applications}.
International Mathematics Research Notices, 2015,{17}: 7608–7619.

\bibitem[Re]{Re}
Reilly, R. C. 
{Geometric applications of the solvability of Neumann problems on a Riemannian manifold}. 
Archive for Rational Mechanics and Analysis,1980,{75(1)}: 23–9.

\bibitem[Ro]{Ro}
Ros, A. 
{Compact hypersurfaces with constant higher-order mean curvatures}. 
Revista Mathm\'{a}tica Iberoamericana 3, 1987,{3}: 447–53.


\bibitem[ST]{ST}
Shi Y, Tian G. {Rigidity of asymptotically hyperbolic manifolds}. Commun Math Phys, 2005, {259}: 545-559.

\bibitem[WX]{WX}
Wang G, Xia C. {Uniqueness of stable capillary hypersurfaces in a ball}. Mathematische Annalen, 2019,{374}:1845-1882.

\bibitem[WZ1]{WZ1}
Wang F, Zhou H. {Comparison theorems for GJMS operators}. Science China Mathematics, 2020, {64(11)}: 2479–2494.

\bibitem[WZ2]{WZ2}
Wang F, Zhou H. {A note on the compactness of Poincar\'{e}–Einstein manifolds}. Commun. Contemp. Math.,  2022, {2250015}: 35pp.
\end{thebibliography}
\end{document}